\documentclass[
11pt]{amsart}

\usepackage{amsthm, amsfonts, amssymb}
\usepackage[pagebackref,colorlinks]{hyperref}
\usepackage[all]{xy}

\theoremstyle{definition}
\newtheorem{ntn}{Notation}[section]
\newtheorem{dfn}[ntn]{Definition}
\theoremstyle{plain}
\newtheorem{lem}[ntn]{Lemma}

\newtheorem{prp}[ntn]{Proposition}
\newtheorem{thm}[ntn]{Theorem}
\newtheorem{cor}[ntn]{Corollary}

\theoremstyle{remark}

\newtheorem{rem}[ntn]{Remark}

\def\floor[#1]{\lfloor #1 \rfloor }

\newcommand{\z}{\mathbb{Z}}

\newcommand{\q}{\mathbb{Q}}

\newcommand{\lan}{\langle}
\newcommand{\ran}{\rangle}

\newcommand{\GL}{\mathit{{\rm GL}}}
\newcommand{\GM}{\mathit{{\rm GM}}}

\newcommand{\SL}{\mathit{{\rm SL}}}

\newcommand{\SK}{\mathit{{\rm SK}}}

\newcommand{\E}{\mathcal{E}}

\newcommand{\ppp}{\mathfrak{p}}

\renewcommand{\H}{\tilde{H}}

\newcommand{\ff}{\mathcal{F}}
\newcommand{\Ab}{\underline{\mathcal{A}b}}
\newcommand{\Rings}{\underline{\mathcal{R}ings}}

\newcommand{\inc}{{\rm inc}}
\newcommand{\id}{{\rm id}}

\newcommand{\tors}{{{\rm Tor}_1^{\z}}}

\newcommand{\s}{\Sigma}
\newcommand{\si}{\sigma}

\newcommand{\del}{\delta}

\newcommand{\arr}{\rightarrow}
\newcommand{\larr}{\longrightarrow}
\newcommand{\harr}{\hookrightarrow}

\newcommand{\se}{\subseteq}

\newcommand{\mt}{\mapsto}

\newcommand{\fff}{{F^\times}}

\newcommand{\rr}{{R^{\times}}}

\newcommand{\stabe}{{\rm Stab}}
\renewcommand{\char}{{\rm char}}
\newcommand{\diag}{{\rm diag}}
\renewcommand{\ker}{{\rm ker}}
\newcommand{\coker}{{\rm coker}}
\newcommand{\im}{{\rm im}}
\newcommand{\ind}{{\rm ind}}

\newcommand {\mtx}[4]
{\left(
\begin{array}{cc}
#1 & #2   \\
#3 & #4
\end{array}
\right)}

\newtheoremstyle{athm}
  {}
  {}
  {\itshape}
  {}
  {\scshape}
  {}
  {.5em}
  {\thmnote{#3}}
\theoremstyle{athm}

\begin{document}

\title{Bloch-Wigner theorem over rings with many units}
\author{Behrooz  Mirzaii}

\maketitle

\section*{Introduction}

The Bloch-Wigner theorem appears in different areas of mathematics.
It appears in  Number theory, in relation with the dilogarithm function
\cite{bloch2000}. It has a deep connection with the scissor congruence
problem in three dimensional hyperbolic geometry \cite{dupont-sah1982},
\cite{neuman-yang1999}. In algebraic $K$-theory, it appears in
connection with the third $K$-group \cite{suslin1991}, etc.

This theorem was proved by Bloch and Wigner independently and in a
somewhat different form. It asserts the existence of the exact
sequence
\[
0 \arr \q/\z \arr H_3(\SL_2(F), \z) \arr \ppp(F) \arr
\bigwedge\!{}_\z^2 \fff \arr K_2(F) \arr 0,
\]
where $F$ is an algebraically closed field of $\char(F)=0$
and $\ppp(F)$ is the pre-Bloch group of $F$.
A precise description of the groups and the maps involved  will be
explained in this paper. See also \cite{dupont-sah1982} or \cite
{sah1989} for an alternative source.

In his remarkable paper \cite{suslin1991}, Suslin generalized this
theorem to all infinite fields. He proved that for an infinite field
$F$, there is an exact sequence
\[
0 \arr \tors(\mu_F, \mu_F)^\sim \arr K_3(F)^\ind \arr \ppp(F) \arr
(\fff \otimes_\z \fff)_\sigma \arr K_2(F) \arr 0,
\]
where $\tors(\mu_F, \mu_F)^\sim$ is the unique nontrivial extension
of $\tors(\mu_F, \mu_F)$ by $\z/2$ and $K_3(F)^\ind:=\coker\Big(K_3^M(F)
\arr K_3(F)\Big)$ is the indecomposable part of $K_3(F)$.

It is very desirable to have such a nice result for a much wider
class of rings, such as the class of commutative rings with many
units (see for example \cite{cathelineau2009}). The purpose of this
article is to provide a version of Bloch-Wigner theorem over the
class of rings with many units. To do this, we shall replace
$K_3(R)^\ind$ or $H_0(\rr, H_3(\SL_2(R), \z))$ with a group
``close'' to them. Here we establish the exact sequence
\begin{gather*}
\tors(\mu_R, \mu_R) \arr \H_3(\SL_2(R), \z)
\arr \ppp(R) \arr (\rr \otimes_\z \rr)_\sigma \arr K_2(R) \arr 0,
\end{gather*}
where $\H_3(\SL_2(R), \z)$ is the group
\[
H_3(\GL_2(R), \z)/
\im\Big(H_3(\rr, \z) + \rr\otimes H_2(\rr, \z)\Big).
\]
Here $\im\Big(H_3(\rr, \z) + \rr \otimes
H_2(\rr,\z)\Big) \se H_3(\GL_2(R),\z)$
is induced by the diagonal inclusion of $\rr \times \rr$ in $\GL_2(R)$.
Moreover when $R$ is a domain with many units, e.g. an infinite field, then
the left hand side map in the above exact sequence is injective.

When we kill two torsion elements in $\H_3(\SL_2(R), \z)$
or when $\rr=\rr^2$, we get a familiar group, that is
\[
\H_3(\SL_2(R), \z[1/2]) \simeq H_0(\rr, H_3(\SL_2(R), \z[1/2]))
\]
or $\H_3(\SL_2(R), \z)\simeq H_3(\SL_2(R), \z)$, respectively.
We should also mention that,
\[
\H_3(\SL_2(R), \q)\simeq H_0(\rr, H_3(\SL_2(R), \q))
\simeq K_3(R)^\ind \otimes \q
\]
(see  \cite{elbaz1998} or \cite{mirzaii-2008} for the second isomorphism).

Let $k$ be an algebraically closed field of $\char(k)\neq 2$.
If $R$ is the ring of dual numbers $k[\epsilon]$
or a local henselian ring with residue field $k$,
then we have the Bloch-Wigner exact sequence
\[
\tors(\mu_R, \mu_R) \arr H_3(\SL_2(R), \z)
\arr \ppp(R) \arr  (\rr \otimes_\z \rr)_\sigma
\arr K_2(R) \arr 0.
\]
Moreover when $R$ is a henselian domain or a $k$-algebra with $\char(k)=0$,
then the left-hand side map in this exact sequence is injective.
Furthermore if $\char(k)=0$,  $H_3(\SL_2(R), \z)$ reduces to $K_3(R)^\ind$.

In some cases, we also establish a `relative' version of Bloch-Wigner
exact sequence.

To prove our main theorem, we introduce and study a spectral sequence
similar to the one introduced in \cite{mirzaii2007}.
This spectral sequence has a close connection to the one introduced
in \cite[\S 2]{suslin1991}, but it is somewhat easier to study.
Our main result comes out of careful analysis of this spectral sequence.

\subsection*{Notation}
In this paper by $H_i(G)$ we mean  the  homology of group $G$ with
integral coefficients, namely $H_i(G, \z)$. By $\GL_n$ (resp.
$\SL_n$) we mean the general (resp. special) linear group $\GL_n(R)$
(resp. $\SL_n(R)$).
If $A \arr A'$ is
a homomorphism of abelian groups, by $A'/A$ we mean $\coker(A \arr
A')$. We denote an element of $A'/A$ represented by $a' \in A'$ again
by $a'$. If $Q=\z[1/2]$ or $\q$, by $A_Q$ we mean $A \otimes_\z Q$.

\section{Bloch groups over rings with many units}\label{section1}

Let $\rr$ be the multiplicative group of invertible
elements of a commutative ring $R$.
Define the {\it pre-Bloch group} $\ppp(R)$ of $R$ as the quotient of
the free abelian group $Q(R)$ generated by symbols $[a]$, $a, 1-a \in
\rr$, to the subgroup generated by elements of the form
\[
[a] -[b]+\bigg[\frac{b}{a}\bigg]-\bigg[\frac{1- a^{-1}}{1- b^{-1}}\bigg]
+ \bigg[\frac{1-a}{1-b}\bigg],
\]
where $a,1-a, b, 1-b, a-b \in \rr$. Define
\[
\lambda': Q(R) \arr \rr \otimes \rr, \ \ \ \ [a] \mt a \otimes (1-a).
\]
\begin{lem}\label{suslin-lemma}
If $a,1-a, b, 1-b, a-b \in \rr$, then
\[
\lambda'\Big(
[a] -[b]+\bigg[\frac{b}{a}\bigg]-\bigg[\frac{1- a^{-1}}{1- b^{-1}}\bigg]
+ \bigg[\frac{1-a}{1-b}\bigg] \Big)
=a \otimes \bigg( \frac{1-a}{1-b}\bigg)+\bigg(\frac{1-a}{1-b}\bigg)\otimes a.
\]
\end{lem}
\begin{proof}
This follows from a direct computation.
\end{proof}

Let $(\rr \otimes \rr)_\sigma :=\rr \otimes \rr/
\lan a\otimes b + b\otimes a: a, b \in \rr \ran$.
We denote the elements of $\ppp(R)$ and $(\rr \otimes \rr)_\sigma$
represented by $[a]$ and $a\otimes b$ again by $[a]$ and $a\otimes b$,
respectively. By Lemma \ref{suslin-lemma}, we have a well-defined map
\[
\lambda: \ppp(R) \arr (\rr \otimes \rr)_\sigma, \ \ \
[a] \mt a \otimes (1-a).
\]

\begin{dfn}
The kernel of $\lambda: \ppp(R) \arr (\rr \otimes \rr)_\sigma$ is called the
{\it Bloch group} of $R$ and is denoted by $B(R)$.
\end{dfn}

\begin{rem}
This version of the pre-Bloch group and the Bloch group of a commutative
ring is due to Suslin
\cite[\S 1]{suslin1991}. We refer to \cite{lichenbaum1989} and \cite{sah1989}
for more information in this direction.
\end{rem}

In this article we will study Bloch groups over
the class of rings with many units, which seems the
appropriate class of rings to study them.

\begin{dfn}
We say that a commutative ring $R$ is a {\it ring with many units}
if for any $n \ge 2$ and for any finite number of surjective linear
forms $f_i: R^n \arr R$, there exists a $v \in  R^n$ such that, for
all $i$, $f_i(v) \in \rr$.
\end{dfn}

For a commutative ring $R$, the $n$-the Milnor $K$-group $K_n^M(R)$ is
defined as an abelian group generated by symbols
\[
\{a_1, \dots, a_n\}, \ \ \ \  a_i \in \rr,
\]
subject to multilinearity and the relation $\{a_1, \dots, a_n \}=0$
if there exits $i, j$, $i \neq j$, such that $a_i+a_j=0$ or  $1$.

\begin{lem}\label{milnor-k2}
Let $R$ be a commutative ring with many units. Then
\[
K_2^M(R) \simeq \rr \otimes \rr/\lan a \otimes
(1-a):a, 1-a \in \rr \ran.
\]
\end{lem}
\begin{proof}
See Proposition $3.2.3$ in \cite{guin1989}.
\end{proof}

Therefore for a ring $R$ with many units we obtain the exact sequence

\[
0 \arr B(R) \arr \ppp(R) \overset{\lambda}{\arr} (\rr \otimes \rr)_\sigma
\arr K_2^M(R) \arr 0.
\]

The study of rings with many units is originated by  W. van der
Kallen in \cite{vdkallen1977}, where he showed that $K_2$ of such
rings behaves very much like $K_2$ of fields.
In fact he proved
that when $R$ is a ring with many units, then
\[
K_2(R)\simeq K_2^M(R).
\]
Here we give a new proof of this fact (Corollary \ref{quillen-k2}).
See \cite[Corollary 4.3]{nes-suslin1990} for another proof of this fact.

Important examples of rings with many units are semilocal rings
which their residue fields are infinite. In particular for an
infinite  field $F$, any commutative finite dimensional $F$-algebra
is a semilocal ring and so it is a ring with many units.

\begin{rem}
Let $R$ be a commutative ring with many units.
It is easy to see that for any $n \ge 1$,
there exist $n$ elements in $R$ such that the sum of each nonempty subfamily
belongs to $\rr$ \cite[Proposition 1.3]{guin1989}. Rings with this
property are considered by Nesterenko and Suslin in
\cite{nes-suslin1990}. For more about rings with many units we refer
to \cite{vdkallen1977}, \cite{nes-suslin1990}, \cite{guin1989} and
\cite{mirzaii2008}.
\end{rem}

In the rest of this article we always assume that $R$ is a
commutative ring with many units, unless it is mentioned otherwise.

\section{Homological interpretation of Bloch group}\label{homological-bloch}

Let $C_l(R^2)$ be the free abelian group with
a basis consisting of $(l+1)$-tuples $(\lan v_0\ran, \dots, \lan v_l\ran)$,
where every $\min\{l+1, 2\}$ of $v_i \in R^2$ is a
basis of a direct summand of $R^2$.
Here by $\lan v_i\ran$ we mean the submodule of $R^2$ generated by $v_i$.
We consider $C_{l}(R^2)$ as a left $\GL_2$-module in a natural way.
If it is necessary we convert this action to the right action
by the definition $m.g:=g^{-1}m$. Let us define a differential operator
\[
\partial_l : C_l(R^2) \arr C_{l-1}(R^2), \ \ l\ge 1,
\]
as an alternating sum of face operators $d_i$, which throw away the
$i$-th component of generators and let
$\partial_0: C_0(R^2) \arr \z$, $\sum_i n_i(\lan v_i\ran) \mt \sum_i n_i$.
The complex
\begin{displaymath}
C_\ast \ : \ \ \
\cdots \larr  C_2(R^2) \overset{\partial_2}{\larr}
C_1(R^2) \overset{\partial_1}{\larr}
C_0(R^2) \overset{\partial_0}{\larr} \z \larr 0
\end{displaymath}
is exact. (This follows from \cite[Lemma 1]{mirzaii2008}).
By applying the functor $H_0$ to
\[
C_{4}(R^2)\arr C_{3}(R^2) \arr \partial_{3}(C_{3}(R^2)) \arr 0,
\]
we get the exact sequence
\[
C_{4}(R^2)_{\GL_2} \arr C_{3}(R^2)_{\GL_2} \arr
H_0(\GL_2, \partial_{3}(C_{3}(R^2))) \arr 0.
\]
For simplicity we use the following standard notation
\[
\infty:=\lan e_1 \ran, \ \  0:=\lan e_2 \ran, \ \ 1:=\lan e_1+e_2 \ran,
\ \ b^{-1}:=\lan e_1+be_2 \ran.
\]
Denote the orbit of the frame
$ (\infty, 0, 1, a^{-1}) \in C_{3}(R^2)$
by $p(a)$ and the orbit of the frame
$(\infty, 0, 1, a^{-1}, b^{-1}) \in C_{4}(R^2)$ by $p(a, b)$,
where $a, 1-a, b, 1-b, a-b \in \rr$. Then
\begin{gather*}
C_{3}(R^2)_{\GL_2}=\coprod_{a} \z.p(a), \ \ \
C_{4}(R^2)_{\GL_2}=\coprod_{a, b} \z.p(a, b).
\end{gather*}
Now a direct computation shows that
\begin{equation}\label{pre-bloch}
H_0(\GL_2, \partial_{3}(C_{3}(R^2)))\simeq \ppp(R).
\end{equation}

{}From the short exact sequence
$0 \arr \partial_1(C_1(R^2)) \arr C_0(R^2) \overset{\partial_0}{\arr} \z \arr 0$
we obtain the connecting homomorphism
$H_{3}(\GL_2) \arr H_2(\GL_2,\partial_1(C_1(R^2)))$.
By iterating this process, we get a homomorphism
\[
\eta': H_3(\GL_2) \arr H_0(\GL_2, \partial_{3}(C_{3}(R^2)))\simeq \ppp(R).
\]
Since $C_0(R^2) \arr \z$ has a $(\rr \times \rr)$-equivariant section $m \mt m(0)$, the
restriction of $\eta'$ to $H_3(\rr \times \rr)$ is zero. Thus we obtain a homomorphism
\[
\eta: H_3(\GL_2)/H_3(\rr \times \rr) \arr \ppp(R).
\]

\begin{prp}\label{homological-Bloch}
Let $R$ be a ring with many units. Then
$\eta$ is injective and $B(R)$ is its homomorphic image. That is
\[
B(R)\simeq H_3(\GL_2)/H_3(\rr \times \rr).
\]
\end{prp}

We will postpone the proof of this proposition to Section \ref{Hom-Bloch}.

\begin{rem}
Let $\GM_2$ denotes the subgroup of monomial matrices in $\GL_2$.
By a well known result of Suslin, for an infinite field $F$,
\[
B(F)\simeq H_3(\GL_2(F))/H_3(\GM_2(F)),
\]
 \cite[Theorem 2.1]{suslin1991}.
In \cite[Remark 2.1]{suslin1991}, Suslin claims
that most of the times $\im\Big(H_3(\GM_2(F)) \arr H_3(\GL_2(F))\Big)$ is strictly
larger than  $\im\Big(H_3(\fff \times \fff) \arr H_3(\GL_2(F))\Big)$! It appears that
this claim is not true due to Proposition \ref{homological-Bloch} above.
\end{rem}

\section{The spectral sequence}

Consider the complex
\[
0 \larr H_1(X) \overset{j}{\larr} C_1(R^2) \larr C_0(R^2) \larr 0,
\]
where $H_1(X):=\ker(\partial_1)$. (See Remark 1.1 in \cite{mirzaii2007}
for an explanation for the choice of the notation).
Set $L_i=C_i(R^2)$ for  $i=0,1$, and $L_2=H_1(X)$.

Let $F_\ast \arr \z$ be the standard resolution of $\z$ over $\GL_2$
\cite[Chapter I, \S 5]{brown1994}:
\[
\cdots \larr
F_2 \overset{\delta_2}{\larr} F_1
\overset{\delta_1}{\larr} F_0 \overset{\epsilon}{\larr} \z \larr 0.
\]
From  the double complex $C_{\ast,\ast}:=F_\ast \otimes_{\GL_2} L_\ast$,
$C_{p,q}:=F_q \otimes_{\GL_2} L_p$,
one obtains a first quadrant spectral sequence
\[
E_{p, q}^1=\begin{cases}
H_q(\GL_2, C_p(R^2)) & \text{if $p=0, 1$ }\\
H_q(\GL_2, H_1(X)) & \text{if $p=2$ }\\
0 & \text{if $p \ge 3$}\end{cases}
\Rightarrow H_{p+q}(\GL_2).
\]
Using the Shapiro lemma \cite[Chapter III,
Proposition~6.2]{brown1994} and a theorem of Suslin
\cite[Theorem~1.9]{suslin1985}, \cite[2.2.2]{guin1989},
\[
E_{p, q}^1\simeq H_q(\rr \times \rr), \ {\rm for}\  p=0,1.
\]
It is not difficult to see that
\[
d_{1, q}^1=H_q(\alpha) - H_q(\id),
\]
where $\alpha: \rr \times \rr \arr \rr \times \rr$, $(a, b) \mt (b, a)$.

\begin{lem}\label{e-1q}
$E_{1, q}^2=0$.
\end{lem}
\begin{proof}
Consider the following commutative diagram with exact columns
\[
\xymatrix{
F_{q+1} \otimes_{\GL_2}   H_{1}(Y) \ar[r] \ar[d] &
F_{q}   \otimes_{\GL_2}   H_{1}(Y) \ar[r] \ar[d]&
F_{q-1} \otimes_{\GL_2}   H_{1}(Y) \ar[d]\\
F_{q+1} \otimes_{\GL_2}   C_{1}(R^2) \ar[r] \ar[d] &
F_{q}   \otimes_{\GL_2}   C_{1}(R^2) \ar[r] \ar[d] &
F_{q-1} \otimes_{\GL_2}   C_{1}(R^2) \ar[d] \\
F_{q+1} \otimes_{\GL_2}   C_{0}(R^2) \ar[r] &
F_{q} \otimes_{\GL_2}     C_{0}(R^2)\ar[r] &
F_{q-1} \otimes_{\GL_2}   C_{0}(R^2).
}
\]
Let $ x \otimes (\infty, 0)   \in  F_q \otimes_{\GL_2} C_1(R^2)$ represents an
element of $H_q(\GL_2, C_1(R^2)) \simeq  H_q(\rr \times \rr)$ such that
$d_{1,q}^1(\overline{x \otimes (\infty, 0)})=0$. So there exists $y \in F_{q+1}$
such that $\del_{q+1}(y)\otimes (\infty) = x \otimes \partial_1(\infty, 0)=
(xw-x) \otimes (\infty)$, where $w=\mtx 0 1 1 0$. From the isomorphism
\[
F_{q} \otimes_{\stabe_{\GL_2}(\infty)}\z \overset{\simeq}{\larr}
F_{q} \otimes_{\GL_2} C_{0}(R^2), \ \ s \otimes 1 \mt s \otimes (\infty),
\]
we obtain $\del_{q+1}(y) \otimes 1= (xw-x) \otimes 1\in
F_{q} \otimes_{\stabe_{\GL_2}(\infty)}\z$. If
$z=x \otimes \partial_{2}(\infty, 0,1)$, then
\[
(\id_{F_q} \otimes j)(z)=x \otimes (\infty, 0) + (xwg-xg) \otimes (\infty, 0).
\]
where $j:H_{1}(Y) \harr C_{1}(R^2)$ and
$g=\mtx 1 {-1} 0 1$. Consider the natural map
\[
F_{q} \otimes_{\stabe_{\GL_2}(\infty)}\z \larr
F_{q}\otimes_{\GL_2} C_{1}(R^2), \ \ s \otimes 1 \mt s \otimes (\infty, 0).
\]
Since $(xw-x) \otimes 1=(xwg-xg) \otimes 1$ in
$F_{q} \otimes_{\stabe_{\GL_2}(\infty)}\z$, by the above map we have
$\del_{q+1}(y) \otimes (\infty, 0)=(xwg -xg) \otimes (\infty, 0)
\in F_{q} \otimes_{\GL_2} C_{1}(R^2)$. Therefore
\begin{gather*}
(\id_{F_q} \otimes j)(z)
=x \otimes (\infty, 0)+ \delta_{q+1}(y) \otimes  (\infty, 0).
\end{gather*}
This completes the proof of the triviality of ${E}_{1, q}^2$.
\end{proof}

From the short exact sequence
$0\arr \partial_{3}(C_{3}(R^2)) \arr C_2(R^2)
\overset{\partial_2}{\arr} H_{1}(X) \arr 0$
one obtains the long exact sequence
\begin{gather*}
\cdots \arr H_1(\GL_2,C_2(R^2))
\overset{H_1(\partial_2)}{\larr}  H_1(\GL_2, H_{1}(X)) \arr
H_0(\GL_2, \partial_{3}(C_{3}(R^2))) \\
\arr H_0(\GL_2,C_2(R^2)) \overset{H_0(\partial_2)}{\larr}
H_0(\GL_2, H_{1}(X)) \arr 0.
\end{gather*}

It is easy to see that $H_0(\GL_2, H_{1}(X))\simeq\z$ (see formula (3)
in \cite{mirzaii2007}) and $d_{2,0}^1=\id$. Thus $E_{2,0}^2=0$.
The composition
\begin{gather*}
\hspace{-4 cm}
\rr \simeq H_1(\GL_2,C_2(R^2)) \overset{H_1(\partial_2)}{\larr}
H_1(\GL_2, H_{1}(X))\\
\hspace{6 cm}
\overset{H_1(j)}{\larr} H_1(\GL_2,C_{1}(R^2))\simeq \rr \times \rr
\end{gather*}
is given by $H_1(j)\circ H_1(\partial_2)(a)=(a, a)$.
Thus $H_1(\partial_2)$ is injective.
Since $H_0(\GL_2,C_2(R^2))\simeq \z$ and $H_0(\partial_2)=\id$,
the above results, together with (\ref{pre-bloch}),
imply the exact sequence
\begin{equation}\label{mirzaii-exact}
0 \arr \rr \arr H_1(\GL_2, H_{1}(X)) \arr \ppp(R) \arr 0.
\end{equation}
By Lemma \ref{e-1q},
$H_1(\GL_2, H_1(X)) \arr \ker(d_{1,1}^1) \simeq \rr$ is surjective
and one can see without any difficulty that this map splits the
exact sequence (\ref{mirzaii-exact}). Therefore
\begin{equation}\label{mirzaii-split}
H_1(\GL_2, H_{1}(X))\simeq \ppp(R) \oplus \rr.
\end{equation}

Thus we proved the following.

\begin{lem}\label{e-21}
$E_{2, 1}^2\simeq \ppp(R)$.
\end{lem}

In order to describe a map $\ppp(R) \arr H_1(\GL_2, H_{1}(X))$ that splits
the exact sequence (\ref{mirzaii-exact}),  we look at the
following commutative diagram with exact rows
\[
\begin{array}{cccccccc}
0 \arr \!\! &\!\! F_1 \otimes_{\GL_2} \partial_3(C_3(R^2)) \!\!
&\!\! \larr\!\!& \!\! F_1  \otimes_{\GL_2} C_2(R^2)\!\! &  \!\!
\overset{ 1 \otimes \partial_2 }{\larr} \!\! &
\!\! F_1 \otimes_{\GL_2} H_1(X) \!\! & \!\! \arr 0 \\
&\Bigg\downarrow\vcenter{%
\rlap{$\scriptstyle{}$}}
&       & \Bigg\downarrow\vcenter{%
\rlap{$\scriptstyle{ \delta_1} \otimes 1 $}}
&  &\Bigg\downarrow{%
\rlap{$\scriptstyle{}$}} \\
0  \arr \!\! & \!\! F_0 \otimes_{\GL_2}  \partial_3(C_3(R^2))
\!\!&\!\! {\larr}\!\!& \!\! F_0 \otimes_{\GL_2} C_2(R^2)\!\!
&  \!\! \larr \!\!& \!\! F_0 \otimes_{\GL_2}H_1(X) \!\!  & \!\! \arr 0.
\end{array}
\]
The element $[a] \in \ppp(R)$ comes from
\begin{gather*}
\hspace{-2 cm}
x_a:=(1) \otimes \partial_3(\infty, 0, 1, a^{-1})
 \in F_0
\otimes_{\GL_2} \partial_3\Big(C_3(R^2)\Big) \\
\ \ \
= \bigg[(g_1)-(g_2)+(g_3)-(1)\bigg] \otimes (\infty, 0, 1)
\in F_0\otimes_{\GL_2} C_2(R^2),
\end{gather*}
where
\[
g_1=\mtx 0 1 {a-1} 1, \ \ g_2=\mtx  {1-a} a 0 a ,\ \ g_3= \mtx 1 0 0 a.
\]
The element
\[
y_a:=  \bigg[(g_2, g_1)-(g_3, 1)\bigg] \otimes (\infty, 0, 1)
\in F_1\otimes_{\GL_2} C_2(R^2)
\]
maps to $x_a$. And finally $y_a$ maps to
\[
z_a:=
\bigg[(g_2, g_1)-(g_3, 1)\bigg] \otimes \partial_2(\infty, 0, 1)
\in F_1 \otimes_{\GL_2} H_1(X).
\]
Thus a splitting map $\ppp(R) \arr  H_1(\GL_2, H_{1}(X))$
can be given by sending $[a]$ to the
element of $H_1(\GL_2, H_{1}(X))$ represented by $z_a$.

Putting all these together, the ${E}^2$-terms of our
spectral sequence look as follow
\begin{gather*}
\begin{array}{lccccc}
\ast        &       &          &     &  & \\
E_{0,3}^2   &   0   &  \ast    &      & &\\
E_{0,2}^2   &   0   &  \ast    &  0  &  &\\
\rr         &   0   &  \ppp(R) &  0  & 0 &\\
\z          &   0   &  0       &  0  & 0 & .
\end{array}
\end{gather*}
By the K\"unneth theorem,
$H_2(\rr \times \rr) \simeq H_2(\rr) \oplus H_2(\rr) \oplus \rr \otimes \rr$.
An easy calculation shows
\begin{gather*}
d_{1, 2}^1: E_{1, 2}^1=H_2(\rr \times \rr) \arr E_{0, 2}^1= H_2(\rr \times \rr)\\
\ \ \ \ \ \ \ \ \ \ \
(r, s, a \otimes b) \mt (-r+s, r-s, -b \otimes a-a \otimes b).
\end{gather*}
Therefore $E_{0, 2}^1 \simeq H_2(\rr) \oplus (\rr \otimes \rr)_\sigma$.
On the other hand, since for an abelian group $A$, $H_2(A)\simeq \bigwedge{\!\!{}^2} A$
(see Lemma \ref{sublemma} below), we have
\[
E_{0,2}^2 \simeq \bigwedge{\!\!{}^2}(\rr \times \rr)/K,
\]
where
$K=\lan (b,a)\wedge (d,c) -(a,b)\wedge (c,d)|a, b, c,d \in \rr \ran$.
Now an isomorphism
\begin{equation}\label{E-02}
H_2(\rr) \oplus (\rr \otimes \rr)_\sigma \overset{\simeq}{\larr}
\bigwedge{\!\!{}^2}(\rr \times \rr)/K
\end{equation}
can be given by
\[
(a\wedge b, c\otimes d) \mt (a,1)\wedge (b,1)+ (c,1)\wedge (1,d).
\]

Here we introduce a notion that will be used later. Let $G$ be a group and let
\[
{\rm \bf{c}}({g}_1, {g}_2,\dots, {g}_n):=\sum_{\si \in \s_n}
{{\rm sign}(\si)}[{g}_{\si(1)}| {g}_{\si(2)}|\dots|{g}_{\si(n)}]
\in H_n(G),
\]
where ${g}_i \in G$ pairwise commute and $\s_n$ is the symmetric
group of degree $n$. Here we use the bar resolution of $G$
\cite[Chapter I, Section 5]{brown1994} to define the homology of $G$.

\begin{lem}\label{sublemma}
Let $G$ and $G'$ be two groups.
\par {\rm (i)} If $h_1\in G$ commutes with all the elements
$g_1, \dots, g_n \in G$, then
\[
{\rm \bf{c}}(g_1h_1, g_2,\dots, g_n)=
{\rm \bf{c}}(g_1, g_2,\dots, g_n)+{\rm \bf{c}}(h_1, g_2,\dots, g_n).
\]
\par {\rm (ii)}
For every $\sigma \in \s_n$,
${\rm \bf{c}}(g_{\sigma(1)},\dots, g_{\sigma(n)})={\rm sign(\sigma)}
{\rm \bf{c}}(g_1, g_2,\dots, g_n)$.
\par {\rm (iii)}
The shuffle product of ${\rm \bf{c}}(g_1,\dots, g_p)\in H_p(G, \z)$
and ${\rm \bf{c}}(g_1',\dots, g_q') \in H_q(G',\z)$ is
${\rm \bf{c}}((g_1, 1), \dots, (g_p,1),(1,g_1'), \dots, (1,g_q'))
\in H_{p+q}(G \times G')$.
\par {\rm (iv)} If $A$ is an abelian group, then the map
$\bigwedge{\!\!{}^2} A \arr H_2(A)$ given by $a \wedge b \mt {\rm \bf{c}}(a,b)$
is an isomorphism.
\end{lem}
\begin{proof}
The proof is standard.
\end{proof}

\section{Proof of Proposition \ref{homological-Bloch}}\label{Hom-Bloch}

\begin{lem}
The differential map $d_{2,1}^2:\ppp(R) \arr
H_2(\rr) \oplus (\rr \otimes \rr)_\sigma $ is given by
$[a] \mt \bigg( a \wedge (1-a), - a\otimes (1-a) \bigg)$.
\end{lem}
\begin{proof}
Here we argue as in \cite[pp. 189-190]{hutchinson1990} or
\cite[pp 192-193]{dupont-sah1982}.
Consider the following commutative diagram with exact rows
\[
\begin{array}{cccccccc}
 0 \!\! &\!\! \larr \!\! &\!\! F_2 \otimes_{\GL_2} H_1(X) \!\!&\!\! \larr\!\!
& \!\! F_2 \otimes_{\GL_2} C_1(R^2) \!\! &  \!\!
\overset{1 \otimes \partial_1}{\larr} \!\! & \!\!F_2 \otimes_{\GL_2} C_0(R^2) \\
&  &\Bigg\downarrow\vcenter{%
\rlap{$\scriptstyle{}$}}
&       & \Bigg\downarrow\vcenter{%
\rlap{$\scriptstyle{ \delta_2} \otimes 1$}}
&  &\Bigg\downarrow{%
\rlap{$\scriptstyle{}$}} \\
0 \!\! &\!\! \larr \!\! &\!\! F_1 \otimes_{\GL_2} H_1(X)
 \!\!&\!\! \overset{1\otimes j}{\larr}\!\!
& \!\! F_1 \otimes_{\GL_2} C_1(R^2) \!\! &  \!\! \larr \!\!
& \!\! F_1 \otimes_{\GL_2} C_0(R^2).
\end{array}
\]
The map $\ppp(R) \arr H_1(\GL_2, H_1(X))$, constructed in the previous
section, sends $[a]$ to the element represented by
$z_a= \bigg[(g_2, g_1)-(g_3, 1)\bigg] \otimes \partial_2(\infty, 0, 1)$.
Then $(1\otimes j)(z_a)$ is equal to
\[
 \bigg[(g_2 g_1,g_1^2)-(g_3 g_1,g_1)-(g_2^2, g_1 g_2)
+(g_3 g_2, g_2)+(g_2, g_1)-(g_3, 1)\bigg] \otimes (\infty, 0).
\]
Since $a^{-1}g_2 g_1=g_1^2 g_3^{-1}$,
$(a^{-1} g_2^2 g_3^2, g_1 g_2)=  (g_3g_2,g_3g_1){\mtx {a^{-1}(1-a)} 0 0 a }$,
it is easy to see that $ \delta_2 (u_a)\otimes (\infty, 0)=(1\otimes j)(z_a)$,
where
\begin{gather*}
u_a=+(g_3 g_1, g_2, g_1)-(g_3 g_2, g_3 g_1, g_2)
-(a^{-1} g_2^2 g_3^2, g_2^2, g_1 g_2)\\
\hspace{-0.8 cm}
+(a^{-3} g_2^2 g_3^2, a^{-2}g_2^2, 1)
-(a^{-3} g_2^2 g_3^2, a^{-1}g_3^2, 1)\\
\hspace{-1.4 cm}
+ (a^{-1}g_2 g_1, a^{-1}g_1^2, 1)
-(g_1^2g_3^{-1}, ag_3^{-1}, 1) \\
\hspace{-5.3 cm}
+ (g_3, a^{-1}g_3^2, 1).
\end{gather*}
(Hear by $ag$, $a \in \rr$ and $g \in \GL_2$, we mean $\mtx a {0} {0} a g$.)
We have
\[
u_a \otimes \partial_1(\infty, 0) =(u_aw-u_a)\otimes (\infty) \in
F_2 \otimes_{\GL_2} C_0(R^2),
\]
where $w=\mtx 0 1 1 0$. The element of
$E_{0,2}^2 = H_2(\rr \times \rr)/\im(d_{1,2}^1)$
represented by $(u_a w-u_a)\otimes (\infty)$ is $d_{2,1}^2([a])$.
We recall  the isomorphism
\begin{equation}\label{h2b}
H_2(\GL_2, C_0(R^2))\overset{\simeq}{\larr} H_2(B),
\end{equation}
where
$B:=\stabe_{\GL_2}(\infty)=\mtx \rr R 0 \rr$. This
is described  on the chain  level by
\[
F_\ast \otimes_{\GL_2} C_0(R^2) \arr F_\ast \otimes_{B} \z, \ \
y \otimes (\infty) \mt y \otimes 1.
\]
Let $F_\ast(B)\arr \z$ be the standard resolution of $\z$ over $B$.
An augmented preserving chain map of $B$-resolutions
$\phi_\ast: F_\ast \arr F_\ast(B)$
is obtained as follows:\\
Let $s:\GL_2/B \arr \GL_2$ be any (set-theoretic) section of the
canonical projection $\pi: \GL_2 \arr \GL_2/B$. For $g \in \GL_2$, set
$\overline{g}=(s\circ\pi(g))^{-1}g$. Then
\[
\phi_n(g_0, \dots, g_n)=(\overline{g_0}, \dots, \overline{g_n}).
\]
By choosing the section
\[
s(gB)=\begin{cases}
1 & \text{if $g(\infty)=\infty$}\\
w & \text{if $g(\infty)=0$ }\\
{\mtx 1 0 {b} 1} & \text{if $g(\infty)=b^{-1}$,}
\end{cases}
\]
we have
\[
\overline{g}=\begin{cases}
g & \text{if $g(\infty)=\infty$}\\
wg & \text{if $g(\infty)=0$ }\\
{\mtx 1 0 {-b} 1}g & \text{if $g(\infty)=b^{-1}$.}
\end{cases}
\]
Thus on the chain level we have the following map
\[
F_n \otimes_{\GL_2} C_0(R^2) \arr F_n(B) \otimes_{B} \z, \ \
(g_0, \dots, g_n) \otimes (\infty)  \mt
(\overline{g_0}, \dots, \overline{g_n}) \otimes 1,
\]
which induces the isomorphism (\ref{h2b}) on homology.
The diagonal inclusion $\rr \times \rr \arr B$
splits by the map $p: B \arr \rr \times \rr$, $\mtx a c 0 b \mt (a, b)$.
This induces an isomorphism
$H_2(B) \overset{\simeq}{\arr} H_2(\rr \times \rr)$,
which is inverse to the isomorphism $H_2(\rr \times \rr)\overset{\simeq}{\arr}
H_2(B)$, obtained by a theorem of Suslin  \cite[Theorem 1.9]{suslin1985}.

Let $B_\ast(G) \arr \z$ and $F_\ast(G) \arr \z$ be the bar
and the standard resolutions of $\z$ over a group $G$,
respectively. Then the map
\[
F_n(G) \arr B_n(G), \ \
(g_0, \dots, g_n) \mt
[g_0g_1^{-1}|g_1g_2^{-1}|\dots|g_{n-2}g_{n-1}^{-1}|g_{n-1}g_{n}^{-1}]
\]
induces the identity map of $H_\ast(G)$.
Now if we follow the isomorphisms
\begin{gather*}
H_2(\GL_2, C_0(F^2)) \overset{\simeq}{\larr}
H_2(F_\ast(B) \otimes_B \z) \overset{\simeq}{\larr}
H_2(F_\ast(\rr \times \rr) \otimes_{(\rr \times \rr)} \z)\\
\overset{=}{\larr} H_2(B_\ast(\rr \times \rr)
\otimes_{(\rr \times \rr)} \z)=H_2(\rr \times \rr),
\end{gather*}
then by a direct computation one sees that $(u_aw-u_a)\otimes (\infty)$ maps to
\begin{gather*}
\begin{array}{rl}
X_a=
& \!\!\! -[(-a, a^{-1})|(-1, a)]+ [(a^{-1}, a)|(a, 1)]+[(-a^{-1}, a^2)|(-a, a^{-1})]\\
& \!\!\! -[(a, 1)|(a^{-1}, a)]+[(a^{-1},a)|(-1, a)]-[(a, a^{-1})|(1, a)] \\
& \!\!\! -[(a^{-1}, a)|(a^{-2}(1-a)^2, 1)]+[(a, a^{-1})|(a^{-1}, -a^{-1}(1-a)^2)]\\
& \!\!\! +[(a^{-2}(1-a)^2, 1)|(a^{-1}, a)]-[(a^{-1},-a^{-1}(1-a)^2)|(a, a^{-1})]\\
& \!\!\! -[(a, 1)|(a^{-1}(a-1), a^{-1}(a-1))]+[(1, a)|(a^{-1}, -a^{-1}(a-1)^2)]\\
& \!\!\! +[(a^{-1}(a-1), a^{-1}(a-1))|(a, 1)]-[(a^{-1}, -a^{-1}(a-1)^2)|(1, a)]\\
& \!\!\! -[(a, 1)|(a^{-1}, a)]+[(1, a)|(a, a^{-1})].
\end{array}
\end{gather*}
Using the fact that
\begin{gather*}
\begin{array}{rl}
\!
\delta_3([(a^{-1},a)|({-1},a)|(-a,a^{-1})])\! =
& \!\!\!\!\!  [(-1, a)|(-a, a^{-1})]+[(a^{-1}, a)|(a, 1)]-\\
& \!\!\!\!\!  [(-a^{-1}, a^2)|(-a, a^{-1})]-[(a^{-1}, a)|(-1, a)],
\end{array}
\end{gather*}
we see that
\begin{gather*}
\begin{array}{rl}
X_a=
& \!\!\! +{\rm \bf{c}}((-1, a),(-a, a^{-1})) + 2 {\rm \bf{c}}((a^{-1}, a),(a, 1))
+{\rm \bf{c}}((1, a),(a, a^{-1})) \\
&  \!\!\! +{\rm \bf{c}}((a^{-2}(1-a)^2, 1),(a^{-1}, a))
+{\rm \bf{c}}((a, a^{-1}),(a^{-1}, -a^{-1}(1-a)^2)) \\
&  \!\!\! +{\rm \bf{c}}((a^{-1}(a-1), a^{-1}(a-1)),(a, 1))
+{\rm \bf{c}}((1, a),(a^{-1}, -a^{-1}(a-1)^2)).
\end{array}
\end{gather*}
By an easy computation we see that under the map
\[
H_2(\rr \times \rr) \arr E_{0,2}^2\simeq \bigwedge{\!\!{}^2}(\rr \times \rr)/K ,
\]
$X_a$ maps to $(a,1)\wedge(1-a,1)-(a,1)\wedge(1,1-a)$ (see Lemma \ref{sublemma}).
Now the claim follows from the isomorphism (\ref{E-02}).
\end{proof}

As a corollary, we give a homological proof of a theorem of
W. van der Kallen \cite{vdkallen1977}.

\begin{cor}\label{quillen-k2}
Let $R$ be a commutative ring with many units. Then
\[
K_2(R)\simeq K_2^M(R).
\]
\end{cor}
\begin{proof}
By an easy analysis of the above spectral sequence,
and  Lemma \ref{milnor-k2}, one sees that
\[
H_2(\GL_2) \simeq E_{0,2}^\infty \simeq  H_2(\rr) \oplus (\rr \otimes \rr)_\sigma/
\lan (a\wedge (1-a), -a \otimes (1-a))| a \in \rr \ran.
\]
From the maps
\begin{gather*}
\begin{array}{ll}
H_2(\rr) \arr E_{0,2}^\infty, & x \mt (x, 0), \\
E_{0,2}^\infty \arr H_2(\rr), & (x, c\otimes d) \mt x+c\wedge d,\\
K_2^M(R) \arr E_{0,2}^\infty, & \{a,b\} \mt (a \wedge b,- a \otimes b), \\
E_{0,2}^\infty \arr K_2^M(R), & (x, c\otimes d) \mt -\{c,d\},
\end{array}
\end{gather*}
we obtain the decomposition
\[
H_2(\GL_2) \simeq  H_2(\rr) \oplus K_2^M(R).
\]

From the
corresponding Lyndon-Hochschild-Serre spectral sequence of the extension
$1 \arr \SL \arr  \GL \overset{\det}{\arr} \rr \arr 1$, using
the fact that $\SK_1(R):=\SL/E(R)=0$ \cite[Proposition 2]{mirzaii2008},
we obtain the decomposition
\[
H_2(\GL) \simeq H_2(\rr) \oplus K_2(R).
\]
Here $E(R)$ is the elementary
subgroup of $\GL$. To obtain the above
decomposition we use the fact that
$E(R)$ is a perfect group, and  $K_2(R)=H_2(E(R))$.
Now the claim follows from the homology
stability theorem $H_2(\GL_2)\simeq H_2(\GL)$
\cite[Theorem 1]{guin1989}.
\end{proof}

Now we are ready to prove Proposition \ref{homological-Bloch}.

~
\\
{\it Proof of Proposition}  \ref{homological-Bloch}.
The spectral sequence studied in the above,
implies the exact sequence
\[
E_{0,3}^2 \arr H_3(\GL_2) \arr B(R) \arr 0.
\]
One can show without difficulty that  the  map $H_3(\GL_2) \arr B(R)$
coincides with the one obtained in Section \ref{homological-bloch}.
The rest follows from the fact that
$E_{0,3}^2=H_3(\rr \times\rr)/\im(d_{1,3}^1)$.
$\hspace{8.5 cm} \Box$

\section{Bloch-wigner exact sequence}

We are now ready to prove our main theorem, which can be considered as a
generalization of the known Bloch-Wigner theorem.

\begin{thm}\label{bloch-wigner1}
Let $R$ be  a commutative ring with many units.
We have the exact sequence
\[
\tors(\mu_R, \mu_R) \arr \H_3(\SL_2(R), \z)\arr B(R) \arr 0,
\]
where
$\H_3(\SL_2(R), \z):=H_3(\GL_2)/\im\Big(H_3(\rr) \oplus \rr \otimes H_2(\rr)\Big)$.

When $R$ is an integral domain, then the left hand side map in the above
exact sequence is injective.
\end{thm}
\begin{proof}
Consider the exact sequence
\[
E_{0,3}^2 \arr H_3(\GL_2) \arr B(R) \arr 0
\]
(see the proof of Proposition \ref{homological-Bloch}
in the previous section). To describe $E_{0,3}^2$, let
$M=H_3(\rr) \oplus H_3(\rr) \oplus \rr \otimes H_2(\rr) \oplus
H_2(\rr) \otimes \rr.$
We have the following commutative diagram with exact rows
\[
\begin{array}{ccccccccc}
0  & \larr & M & {\larr}&   H_3(\rr \times \rr) &   \larr
& \tors(\mu_R, \mu_R) & \larr &  0 \\
&  &\Bigg\downarrow\vcenter{%
\rlap{$\scriptstyle{d_{1,3}^1|_M}$}}
&       & \Bigg\downarrow\vcenter{%
\rlap{$\scriptstyle{d_{1,3}^1}$}}
&  &\Bigg\downarrow{%
\rlap{$\scriptstyle{\nu}$}}&      &  \\
0  & \larr & M & \larr &  H_3(\rr \times \rr)
&   \larr   &    \tors(\mu_R, \mu_R) & \larr &  0.
\end{array}
\]
Here the exact rows follow from the K\"unneth theorem and
$\nu$ is induced in a natural way. We also have
\begin{gather*}
\hspace{-5.3 cm}
d_{1,3}^1|_M:
(r, s, a \otimes r', s'\otimes b ) \mt \\
\ \ \ (-r+s,r-s, -a \otimes r'+ b \otimes s',
r'\otimes a - s' \otimes b).
\end{gather*}
So the Snake lemma implies the exact sequence
\begin{equation}\label{exact-tor}
N \arr E_{0,3}^2 \arr
\tors(\mu_R, \mu_R)/\im(\nu) \arr 0,
\end{equation}
where $N \simeq H_3(\rr) \oplus \rr \otimes H_2(\rr)$. From the
commutative diagram
\[
\begin{array}{ccccccc}
N & \overset{=}{\larr} &   N &  &  &  &   \\
\Bigg\downarrow\vcenter{%
\rlap{$\scriptstyle{}$}}
&       & \Bigg\downarrow\vcenter{%
\rlap{$\scriptstyle{}$}}&  & &  &  \\
E_{0,3}^2 \!\! &\!\! \larr \!\! & \!\! H_3(\GL_2)\!\!
& \!\! \larr \!\!  &  \!\!B(R)\!\! & \!\! \larr \!\! & \!\! 0,
\end{array}
\]
we obtain the exact sequence
\[
E_{0,3}^2/N \arr H_3(\GL_2)/N \arr B(R) \arr 0.
\]
Now the first claim follows from the exact sequence (\ref{exact-tor}).
To complete the proof, for an integral domain $R$
we must prove the injectivity of
\[
\tors(\mu_R, \mu_R) \arr \H_3(\SL_2).
\]
Denote by $\overline{F}$ the algebraic closure of the quotient field $F$ of $R$.
Since the homomorphism
\[
\tors(\mu_R, \mu_R)\arr \tors(\mu_{\overline{F}}, \mu_{\overline{F}})
\]
is injective, we may replace $R$ with the field $\overline{F}$.
We know that for an algebraically closed field $F$,
$H_3(\SL_2(F))\simeq \H_3(\SL_2(F))$ (see the proof of Corollary
\ref{bloch-wigner2} below) and $K_3(F)^\ind \simeq H_3(\SL_2(F))$
\cite[Proposition 6.4]{mirzaii-2008}. But torsion elements of
$K_3(F)^\ind$ are known by a result of Suslin \cite{suslin1983}, \cite{suslin1984}.
\end{proof}

\begin{rem}
(i) The main results of this paper,
Proposition \ref{homological-Bloch} and Theorem \ref{bloch-wigner1},
remain true over non-commutative rings with many units.

Let $R$ be a non-commutative ring with many units
(see \cite[\S 2]{mirzaii-2008}, \cite[\S 1]{guin1989}). Let $\ppp(R)$ be
the quotient of the free abelian group $Q(R)$, defined as in Section \ref{section1},
to the subgroup generated by
\begin{gather*}
[aba^{-1}]-[b], \\
[a]-[b]+[a^{-1}b]-[{(1- b^{-1})}^{-1}(1- a^{-1})]
+[{(1- b)}^{-1}(1- a)],
\end{gather*}
where $a,1-a, b, 1-b, a-b \in \rr$. We define $B(R)$  as the
kernel of the map
\[
\lambda: \ppp(R) \arr (K_1(R) \otimes K_1(R))_\sigma, \ \ \
[a] \arr a \otimes (1-a).
\]
Since $K_2^M(R)\simeq (K_1(R) \otimes K_1(R))_\sigma/\lan a \otimes (1-a)
| a, 1-a \in K_1(R) \ran$,
we have the exact sequence
\[
0 \arr B(R) \arr \ppp(R) \overset{\lambda}{\arr}
(K_1(R) \otimes K_1(R))_\sigma \arr K_2^M(R) \arr 0.
\]
Note that by the homology stability theorem
\cite[Theorem 1]{guin1989}, $K_1(R)\simeq \rr/[\rr, \rr]$.
Moreover Lemmas \ref{e-1q} and \ref{e-21} are valid,
$E_{p, 0}^2=0$, $p \ge 0$, and
\[
H_1(\GL_2, H_1(X)) \simeq \ppp(R) \oplus K_1(R).
\]
We also have $E_{0,2}^2\simeq H_2(\rr) \oplus
(K_1(R) \otimes K_1(R))_\sigma$ and
\[
d_{2,1}^2:\ppp(R) \arr
H_2(\rr) \oplus (K_1(R) \otimes K_1(R))_\sigma
\]
is given by $[a] \mt \bigg( {\rm \bf{c}}(a,1-a), - a \otimes (1-a) \bigg)$.
Thus we have the isomorphism
$B(R)\simeq H_3(\GL_2)/H_3(\rr \times \rr)$ and  the Bloch-Wigner
exact sequence
\begin{gather*}
\tors({K_1(R)}, {K_1(R)}) \arr \H_3(\SL_2(R), \z)
\arr B(R) \arr 0.
\end{gather*}
%
\par (ii) For a non-commutative ring $R$,
$K_2^M(R)$ and $K_2(R)$ are not isomorph in general.
For more information in this direction see \cite[Theorem 4.2.1]{guin1989}.
\end{rem}

Let $\rr \otimes H_2(\rr) \arr H_3(\GL_2)$, $a \otimes (b \wedge c) \mt
a \cup (b \wedge c)$ be the shuffle product. Note that
\begin{gather*}
a \cup (b \wedge c)=
{\bf c}(\diag(a, 1),\diag(1,b), \diag(1, c)).
\end{gather*}
Let $\Phi: \rr \otimes K_2^M(R) \arr H_3(\GL_2)$ be defined by
\begin{gather*}
a \otimes \{b , c\} \mt
{\bf c}(\diag(a, a),\diag(b,1), \diag(c, c^{-1})).
\end{gather*}

\begin{lem}\label{sublemma1}
Let $\inc: \rr \arr \GL_2$ be the natural inclusion. Then in
$H_3(\GL_2)$ we have
\begin{gather*}
\Phi\Big(a \otimes \{b, c\})=H_3(\inc)\Big( {\bf c}(a, b, c) \Big)
-c\cup (a \wedge b) +a\cup (b \wedge c)+ b\cup (a \wedge c), \\
\hspace{-1.6 cm}
2\ a\cup (b \wedge c)=
2\ H_3(\inc)\Big( {\bf c}(a, b, c) \Big)
+ \Phi\Big(b \otimes \{a, c\}- c \otimes \{a, b\}\Big).
\end{gather*}
\end{lem}
\begin{proof}
These are direct computations.
\end{proof}

The following corollary justifies why we call Theorem \ref{bloch-wigner1}
a generalization of the Bloch-Wigner theorem.

\begin{cor}\label{bloch-wigner2}
Let $R$ be a commutative ring with many units.
\par {\rm (i)} Then
\[
\tors(\mu_R, \mu_R)_{\z[1/2]} \arr H_3(\SL_2, \z[1/2])_\rr
\arr B(R)_{\z[1/2]} \arr 0
\]
is exact.
\par {\rm (ii)} If $\rr=\rr^2$, then
\[
\tors(\mu_R, \mu_R) \arr H_3(\SL_2)
\arr B(R) \arr 0
\]
is exact.

Moreover when $R$ is a domain, then the left hand side maps in the above
exact sequences are injective.
\end{cor}
\begin{proof}
{}From the corresponding Lyndon-Hochschild-Serre spectral
sequence of the extension
$1 \arr \SL_2 \arr  \GL_2 \overset{\det}{\arr} \rr \arr 1$,
we obtain the exact sequence
\begin{gather*}
\hspace{-6 cm}
H_2(\rr, H_2(\SL_2)) \overset{\beta}{\arr} H_3(\SL_2)_\rr \arr \\
\hspace{2 cm}
H_3(\GL_2)/H_3(\rr) \arr H_1(\rr, H_2(\SL_2)) \arr 0.
\end{gather*}

(i) The map $\gamma: \rr \times \SL_2 \arr \GL_2$, $(a, g) \mt a g$
induces an isomorphism of homology groups
$H_3(\rr \times \SL_2, \z[1/2])_\rr \overset{\simeq}{\larr}
H_3(\GL_2,\z[1/2])$
(see the proof of Theorem 6.1 in \cite{mirzaii-2008} or proof of
Corollary 1 in \cite{mirzaii2008}). So by the K\"unneth theorem,
$H_3(\SL_2, \z[1/2])_\rr \harr H_3(\GL_2,\z[1/2])$. One can also
see that
\[
H_1(\rr, H_2(\SL_2, \z[1/2])) \simeq H_1(\rr, H_2(\SL, \z[1/2]))
\simeq  (\rr \otimes K_2^M(R))_{\z[1/2]}.
\]
See \cite[Lemma 2]{mirzaii2008} or the proof of Corollary 6.2 in
\cite{mirzaii-2008} for a proof of the first isomorphism. Note that we
use Corollary \ref{quillen-k2} for the second isomorphism.
Thus we obtained the exact sequence
\begin{gather*}
\hspace{-2 cm}
0{\arr} H_3(\SL_2, \z[1/2])_\rr \arr
H_3(\GL_2, \z[1/2])/H_3(\rr, \z[1/2]) \\
\hspace{7 cm}
\arr (\rr \otimes K_2^M(R))_{\z[1/2]} \arr 0.
\end{gather*}
We show that the map
\begin{gather*}
(\rr \otimes K_2^M(R))_{\z[1/2]} \arr
H_3(\GL_2, \z[1/2])/ H_3(\rr, \z[1/2]),\\
a \otimes \{b,c\} \mt \frac{1}{2}
{\bf c}(\diag(a, a),\diag(b,1), \diag(c, c^{-1})),
\end{gather*}
splits the above exact sequence.
In order to verify this, consider the following commutative diagram
with exact rows
\[
\begin{array}{cccccccc}
 0 \arr \!\! &\!\! H_3(\SL_2,\z[1/2])_\rr
 \!\!&\!\! \larr\!\! & \!\! S_2\!\! &  \!\! {\larr} \!\! &
\!\!(\rr \otimes K_2^M(R))_{\z[1/2]} \!\! & \!\! \arr 0 \\
 &\Bigg\downarrow\vcenter{%
\rlap{$\scriptstyle{}$}}
&       & \Bigg\downarrow\vcenter{%
\rlap{$\scriptstyle{}$}}
&  &\Bigg\downarrow{%
\rlap{$\scriptstyle{=}$}} \\
0  \arr \!\! & \!\!H_3(\SL,\z[1/2]) \!\!&\!\!
{\larr}\!\! & \!\! S_3\!\! &  \!\! \larr \!\!
& \!\!(\rr \otimes K_2^M(R))_{\z[1/2]} \!\!  & \!\! \arr 0,
\end{array}
\]
where $S_i:=H_3(\GL_i, \z[1/2])/ H_3(\rr, \z[1/2])$. The second
exact sequence can be obtained similar to the first one.
Consider the homotopy equivalence
\[
B\GL^+ \sim B\rr \times B\SL^+
\]
(see \cite[Lemma 5.3]{suslin1991} or \cite[Remark 2.1]{arlettaz1992}).
The homology stability theorem and the K\"unneth theorem imply
\[
H_3(\GL_3)\simeq H_3(\GL)\simeq H_3(\SL) \oplus
\rr \otimes K_2^M(R) \oplus H_3(\rr).
\]
It is easy to see that the embedding $\rr \otimes K_2^M(R) \arr H_3(\GL)$
is induced by
\[
a \otimes \{b,c\} \mt
{\bf c}(\diag(a, 1, 1), \diag(1, b, 1),\diag(1, c, c^{-1})),
\]
and that it splits the second exact sequence in the above diagram.
The image of ${\bf c}(\diag(a, a),\diag(b,1), \diag(c, c^{-1}))$
in $H_3(\GL_3)$ is equal to
\begin{gather*}
{\bf c}(\diag(a, 1, a^{-1}), \diag(b, b^{-1} , 1),\diag(c, 1, c^{-1}))\\
+ {\bf c}(\diag(a^2, 1, 1), \diag(1, b, 1),\diag(1, c, c^{-1})).
\end{gather*}
The rest is easy to verify.
Now using Lemma \ref{sublemma1}, it is easy to see that
\[
H_3(\SL_2, \z[1/2])_\rr \simeq H_3(\GL_2, \z[1/2])/\im(L_{\z[1/2]}),
\]
where $L:=H_3(\rr) \oplus \rr \otimes H_2(\rr)$.

(ii) The map $\gamma$ induces the short exact sequence
\[
1 \arr \mu_{2,R} \arr \rr \times \SL_2 \arr \GL_2 \arr 1,
\]
and we consider its corresponding spectral sequence,
\[
\E_{p,q}^2=H_p(\GL_2, H_q(\mu_2)) \Rightarrow H_{p+q}(\rr \times \SL_2).
\]
Using the fact that $H_2(\rr)$
and $K_2^M(R)$ are uniquely
2-divisible (see the proof of Proposition 1.2 in \cite{bass-tate1973}),
we have the following $\E^2$- terms:
\begin{gather*}
\begin{array}{lcllll}
\ast       &     &                    &            &             &  \\
H_3(\mu_2) & 0   & \ast               &            &             &  \\
H_2(\mu_2) & 0   & \ast               & \ast       &             &  \\
\mu_2      & 0   & \tors(\mu_R,\mu_{2,R}) & \ast       &  \ast       &  \\
\z         & \rr & H_2(\GL_2)         & H_3(\GL_2) & H_4(\GL_2). &
\end{array}
\end{gather*}
By an easy analysis of this spectral sequence, one sees that,
the kernel of $H_3(\SL_2) \arr H_3(\GL_2)$ is torsion of
exponent $4$. On the other hand
\[
\begin{array}{ll}
H_2(\rr, H_2(\SL_2)) \!\!\!\!&
\simeq H_2(\rr) \otimes H_2(\SL_2) \oplus \tors(H_1(\rr), H_2(\SL_2)) \\
& \simeq H_2(\rr) \otimes K_2^M(R) \oplus \tors(\mu_R, K_2^M(R)).
\end{array}
\]
Since $H_2(\rr)\otimes K_2^M(R)$ and $\tors(\mu_R, K_2^M(R))$
are uniquely $2$-divisible, $\beta$ must be a trivial map.
This implies that $H_3(\SL_2) \arr H_3(\GL_2)/H_3(\rr)$ is
injective. The rest of the proof is similar to the proof of
the part (i), since
\[
H_1(\rr, H_2(\SL_2)) \simeq \rr \otimes K_2^M(R).
\]
Here a splitting map
$\rr \otimes K_2^M(R)\arr H_3(\GL_2)/H_3(\rr)$
can be given by the formula
\begin{gather*}
a \otimes \{b,c\} \mt
{\bf c}(\diag(\sqrt{a}, \sqrt{a}),\diag(b,1), \diag(c, c^{-1})).
\end{gather*}
\end{proof}

\begin{cor}\label{bloch-wigner3}
Let $k$ be an algebraically closed field
of $\char(k)\neq 2$. Let
$R$ be the ring of dual numbers $k[\epsilon]$ or a henselian local
ring with residue field $k$.
\par {\rm (i)} Then we have the following exact sequence
\[
\tors(\mu_R, \mu_R) \arr H_3(\SL_2) \arr B(R) \arr 0.
\]
If $R$ is a henselian domain, then the left hand side map
in this exact sequence is injective.
\par {\rm (ii)} If $\char(k)=0$, then $H_3(\SL_2)\simeq K_3(R)^\ind$.
Furthermore, if $R$ is a $k$-algebra, then we have the exact
sequence
\[
0 \arr \tors(\mu_R, \mu_R) \arr K_3(R)^\ind \arr B(R) \arr 0.
\]
\end{cor}
\begin{proof}
Since $\char(k)\neq 2$, $\rr=\rr^2$ (use Hensel's lemma, when $R$ is
henselian). Thus the first part follows from Theorem \ref{bloch-wigner1} and
Corollary \ref{bloch-wigner2}. The proof of the isomorphism
$H_3(\SL_2)\simeq K_3(R)^\ind$
is similar to the proof of Proposition 6.4(iii) in \cite{mirzaii-2008}.
To prove Proposition 6.4(iii) in \cite{mirzaii-2008}, we used
Theorems 5.4(iii) and 6.1(iii) in \cite{mirzaii-2008}. Note that these theorems
are also valid for rings considered here, with almost the same proofs, once
we make sure that Proposition 4.1 in \cite{mirzaii-2008} is valid as well.
But the proof of this proposition follows {\it mutatis mutandis} as in
\cite{mirzaii-2008}, noting that here $\mu_R\simeq \mu_k$.
If $R$ is the ring of dual numbers $k[\epsilon]$, the isomorphism
$\mu_R\simeq \mu_k$ follows from a direct computation.
If $R$ is a local henselian ring, this follows from Hensel's lemma.
In both cases we need the condition $\char(k)=0$.
To complete the proof, we have to show that if $R$ is a $k$-algebra, then
\[
\tors(\mu_k, \mu_k) \arr K_3(R)^\ind
\]
is injective.
In this case, $K_3(k)^\ind$ is a direct summand of $K_3(R)^\ind$
and the above map factors through $K_3(k)^\ind$. Thus the injectivity claim
can be proved as in the last paragraph of page 189 in \cite{dupont-sah1982}.
This also follows from a result of Suslin in \cite{suslin1991} or
\cite{suslin1984}.
\end{proof}

\begin{rem}
Corollary \ref{bloch-wigner3} remains true
if $k$ is an infinite field such that $k^\times={k^\times}^2$ and $\char(k)\neq 2$.
\end{rem}

\section{Relative bloch-wigner exact sequence}

For a functor $\ff:\Rings \arr \Ab$, from the category of rings
to the category of abelian
groups, we set $\ff(R[\epsilon], \lan
\epsilon \ran):=\ker(\ff(R[\epsilon])
\overset{\epsilon =0}\larr \ff(R))$.

\begin{prp}\label{relative-bloch-wigner}
{\rm (i)} Let $R[\epsilon]$
be the ring of dual numbers, where $R$ is a ring with many units.
Then we have the relative Bloch-Wigner exact sequence
\[
0 \! \arr\! K_3(R[\epsilon], \lan \epsilon \ran)_\q^\ind \arr
\ppp(R[\epsilon], \lan \epsilon \ran)_\q \arr \!\!
\bigwedge{\!\!{}^2}(R[\epsilon], \lan \epsilon \ran)_\q^\times
\arr\! K_2(R[\epsilon], \lan \epsilon \ran)_\q \arr 0.
\]
\par {\rm (ii)} Let k be an algebraically closed field of $\char(k)=0$.
Then we have the relative Bloch-Wigner exact sequence
\[
0  \arr K_3(k[\epsilon], \lan \epsilon \ran)^\ind \arr
\ppp(k[\epsilon], \lan \epsilon \ran) \arr
\bigwedge{\!\!{}^2}(k[\epsilon], \lan \epsilon \ran)^\times
\arr K_2(k[\epsilon], \lan \epsilon \ran) \arr 0.
\]
\end{prp}
\begin{proof}
The ring $R[\epsilon]$ is a ring with many units
\cite[Proposition 2.5]{elbaz1999}.
Consider the following diagram with exact rows

\[
\begin{array}{ccccccccccc}
 0 \!\!\! &\!\!\! \larr \!\!\! &\!\!\! K_3(R[\epsilon])_\q^\ind
 \!\!\!&\!\!\! \larr\!\!\!
& \!\!\!  \ppp(R[\epsilon])_\q\!\!\! &  \!\!\! \larr \!\!\!
&\!\!\! \bigwedge{\!\!{}^2}{R[\epsilon]}_\q^\times\!\!\!
& \!\!\!\larr\!\!\! & \!\!\!K_2(R[\epsilon])_\q\!\!\! &\!\!\! \larr
\!\!\!& \!\!\! 0 \\
&  &\Bigg\downarrow\vcenter{%
\rlap{$\scriptstyle{\epsilon=0}$}}
&       & \Bigg\downarrow\vcenter{%
\rlap{$\scriptstyle{\epsilon=0}$}}
&  &\Bigg\downarrow{%
\rlap{$\scriptstyle{\epsilon=0}$}}
&      &  \Bigg\downarrow{%
\rlap{$\scriptstyle{\epsilon=0}$}}& &\\
 0 \!\!\! & \larr & K_3(R)_\q^\ind & \larr
&   \ppp(R)_\q &   \larr & \bigwedge{\!\!{}^2}R_\q^\times
& \larr & K_2(R)_\q & \larr & \!\!\! 0.
\end{array}
\]
The exact rows follow from Corollary \ref{bloch-wigner2} and
\cite[Proposition 6.4]{mirzaii-2008}, \cite[Theorem 2.2]{elbaz1998}.
The vertical maps are surjective. We can break the above diagram
into two diagrams with rows that are short exact sequences. Now the part (i)
follows from the Snake lemma. The proof of (ii) is analogue.
Here one has to use Corollary \ref{bloch-wigner3}(ii).
\end{proof}

\begin{rem}
(i) For a commutative ring $R$, the relative group
$K_2(R[\epsilon], \lan \epsilon \ran)$ is studied by
W. van der Kallen in terms of generators and relations
\cite{vdkallen1971}. As a special case, he proves that
when $\frac{1}{2} \in R$, then
$K_2(R[\epsilon], \lan \epsilon \ran)\simeq \Omega_{R/\z}^1$.
\par (ii)
The relative group
$K_3(R[\epsilon], \lan \epsilon \ran)_\q^\ind$ is studied in
\cite{elbaz1999} in terms of homology of linear groups.
\par (iii)
For a field $k$, the groups $\ppp(k[\epsilon], \lan \epsilon \ran)$
and $K_3(k[\epsilon], \lan \epsilon \ran)^\ind$ are studied by
Cathelineau in \cite{cathelineau2009}.
\end{rem}

\begin{cor}
Let k be an algebraically closed field of $\char(k)=0$.
Let $R$ be a henselian local $k$-algebra with residue ring
$R/\mathfrak{m} \simeq k$. For a functor $\ff:\Rings \arr \Ab$, define the
relative group
$\ff(R, \mathfrak{m}):=\ker\Big(\ff(R) \arr \ff(R/\mathfrak{m})\Big)$.
Then we have the relative Bloch-Wigner exact sequence
\[
0 \! \arr\! K_3(R,\mathfrak{m})^\ind \arr
\ppp(R, \mathfrak{m})\arr \!\!
\bigwedge{\!\!{}^2}(R, \mathfrak{m})^\times
\arr\! K_2(R, \mathfrak{m}) \arr 0.
\]
\end{cor}
\begin{proof}
The proof is similar to the proof of Proposition \ref{relative-bloch-wigner}.
\end{proof}

\subsection*{Acknowledgements}
Part of this work was done during my stay at IH\'ES and ICTP.
I would like to thank them for their support and hospitality.


\bigskip
\address{{\footnotesize

Department of Mathematics,

Institute for Advanced Studies in Basic Sciences,

P. O. Box. 45195-1159, Zanjan, Iran.

email:\ bmirzaii@iasbs.ac.ir

}}
\end{document}